\documentclass[reqno]{amsart}
\usepackage{lmodern}
\usepackage[T1]{fontenc}
\usepackage[utf8]{inputenc}

\usepackage{amssymb,amsfonts,amsmath, mathtools}
\usepackage[dvipsnames]{xcolor}
\usepackage{enumitem}
\usepackage{csquotes} 
\usepackage{booktabs} 
\usepackage{array}
\usepackage{stackengine}
\usepackage{caption}
\usepackage{bm,bbm}
\usepackage{tikz}
\usepackage{tikz-cd}
\usetikzlibrary{calc}
\usepackage{xargs}            

\definecolor{DefColor}{RGB}{139,30,30}
\definecolor{citeblue}{RGB}{0,4,37}
\definecolor{crefcolor}{RGB}{10,30,139}

\usepackage[pagebackref]{hyperref}
\hypersetup{
    colorlinks=true,
    linkcolor=crefcolor,
    citecolor=PineGreen,
    urlcolor=Plum
}
\usepackage[nameinlink,capitalise,noabbrev]{cleveref}
\renewcommand*{\backref}[1]{}
\renewcommand*{\backrefalt}[4]{%
	\ifcase #1%
	\or (Cited on page~#2.)%
	\else (Cited on pages~#2.)%
	\fi%
}

\newcommand{\mdef}[1]{\textcolor{DefColor}{#1}} 
\newcommand{\tdef}[1]{\textit{\mdef{#1}}}

\newtheorem{thm}{Theorem}[section]
\newtheorem*{thm*}{Theorem}

\newtheorem*{proposition*}{Proposition}
\newtheorem{observation}[thm]{Observation}

\newtheorem{theorem}[thm]{Theorem}
\newtheorem{corollary}[thm]{Corollary}
\newtheorem{proposition}[thm]{Proposition}
\newtheorem{lemma}[thm]{Lemma}

\newtheorem{fact}[thm]{Fact}
\newtheorem*{question*}{Question}

\newtheorem*{theorem*}{Theorem} 
\newtheorem*{conjecture*}{Conjecture}
\theoremstyle{definition}       
\newtheorem{definition}[thm]{Definition}
\newtheorem*{definition*}{Definition}

\newtheorem{example}[thm]{Example}

\newtheorem{remark}[thm]{Remark}

\newcommand{\bD}{\mathbb{D}}

\DeclareMathOperator{\cA}{\mathcal{A}}

\DeclareMathOperator{\cC}{\mathcal{C}}
\DeclareMathOperator{\cD}{\mathcal{D}}
\DeclareMathOperator{\cE}{\mathcal{E}}
\DeclareMathOperator{\cF}{\mathcal{F}}
\DeclareMathOperator{\cG}{\mathcal{G}}

\newcommand{\cK}{\mathcal{K}}

\DeclareMathOperator{\cO}{\mathcal{O}}
\DeclareMathOperator{\cP}{\mathcal{P}}

\newcommand{\cR}{\mathcal{R}}
\newcommand{\cI}{\mathcal{I}}

\newcommand{\cW}{\mathcal{W}}
\newcommand{\cX}{\mathcal{X}}

\DeclareMathOperator{\Cond}{Cond}
\DeclareMathOperator{\ExtrDisc}{ExtrDisc}

\DeclareRobustCommand{\minwidthbox}[2]{%
	\mathmakebox[\ifdim#2<\width\width\else#2\fi]{#1}%
}

\newcommand{\too}[1][]{\xrightarrow{\minwidthbox{#1}{1em}}}

\newcommand{\isoo}{ \xrightarrow{\; \; \simeq \;\;}}
\newcommand{\into}{\hookrightarrow}%

\tikzset{curve/.style={settings={#1},to path={(\tikztostart)
			.. controls ($(\tikztostart)!\pv{pos}!(\tikztotarget)!\pv{height}!270:(\tikztotarget)$)
			and ($(\tikztostart)!1-\pv{pos}!(\tikztotarget)!\pv{height}!270:(\tikztotarget)$)
			.. (\tikztotarget)\tikztonodes}},
	settings/.code={\tikzset{quiver/.cd,#1}
		\def\pv##1{\pgfkeysvalueof{/tikz/quiver/##1}}},
	quiver/.cd,pos/.initial=0.35,height/.initial=0}

\tikzset{between/.style n args={2}{/tikz/execute at end to={
			\tikzset{spath/split at keep middle={current}{#1}{#2}}
}}}

\tikzset{tail reversed/.code={\pgfsetarrowsstart{tikzcd to}}}
\tikzset{2tail/.code={\pgfsetarrowsstart{Implies[reversed]}}}
\tikzset{2tail reversed/.code={\pgfsetarrowsstart{Implies}}}
\tikzset{no body/.style={/tikz/dash pattern=on 0 off 1mm}}

\newcommand{\tCat}{\mathbb{C}\mathrm{at}}

\newcommand{\CatRotimes}{\mathbb{C}\mathrm{at}_R^{\otimes}}

\newcommand{\Catotimes}{\mathbb{C}\mathrm{at}^{\otimes}}
\newcommand{\PrL}{\mathrm{Pr}^L} 

\DeclareMathOperator{\Cat}{Cat}
\DeclareMathOperator{\Catl}{\widehat{\Cat}}

\newcommand{\PrLst}{\mathrm{Pr}^L_{\mathrm{st}}} 
 
\DeclareMathOperator{\Top}{Top}
\DeclareMathOperator{\An}{An}
\DeclareMathOperator{\Spc}{\An}          
\DeclareMathOperator{\Sp}{Sp} 

\DeclareFontFamily{U}{dmjhira}{}
\DeclareFontShape{U}{dmjhira}{m}{n}{ <-> dmjhira }{}

\DeclareMathOperator{\ad}{ad}

\DeclareMathOperator{\corp}{corp}

\DeclareMathOperator{\CMon}{CMon}

\DeclareMathOperator{\CAlg}{CAlg}

\DeclareMathOperator{\Sch}{Sch}
\newcommand{\et}{\text{ét}}

\DeclareMathOperator{\Div}{\Div}

\DeclareMathOperator{\Rig}{Rig}
\DeclareMathOperator{\dbl}{dbl}

\DeclareMathOperator{\Sh}{Sh}

\DeclareMathOperator{\Fun}{Fun}
\DeclareMathOperator{\const}{const}
\DeclareMathOperator{\Mod}{Mod}
\newcommand{\Map}{\operatorname{Map}} 

\newcommand{\colim}{\operatorname*{colim}}
\DeclareMathOperator{\pt}{\ast}

\DeclareMathOperator{\op}{op}

\DeclareMathOperator{\ev}{ev}

\DeclareMathOperator{\id}{id}

\DeclareMathOperator{\Lan}{Lan}

\DeclareMathOperator{\fgt}{fgt}
\newcommand{\co}{\text{co}} 

\newcommand{\laxlim}{\operatorname*{laxlim}}

\newcommand{\laxlimdag}{\operatorname*{laxlim^\dagger}}
\newcommand{\laxcolimdag}{\operatorname*{laxcolim^\dagger}}
\newcommand{\parlaxlim}{\laxlimdag}
\newcommand{\oplaxcolim}{\operatorname*{oplaxcolim}} 
\newcommand{\oplaxcolimdag}    {\operatorname*{oplaxcolim^\dagger}}
\newcommand{\oplaxlim}    {\operatorname*{oplaxlim}}
\newcommand{\oplaxlimdag}    {\operatorname*{oplaxlim^\dagger}} 

\newcommand{\Ar}{\operatorname{Ar}}
\newcommand{\Un}{\operatorname{Un}}
\newcommand{\open}{\mathrm{open}}
\newcommand{\lax}{\mathrm{lax}}
\newcommand{\oplax}{\mathrm{oplax}}
\DeclareMathOperator{\Nat}{Nat}
\newcommand{\ct}{\mathrm{ct}}
\DeclareMathOperator{\CHaus}{CHaus}
\DeclareMathOperator{\Mfld}{Mfld}
\newcommand{\cts}{\mathrm{cts}}
\newcommand{\GAL}{\operatorname{GAL}}
\newcommand{\Gal}{\operatorname{Gal}}
\newcommand{\post}{\mathrm{post}}
\newcommand{\Tw}{\operatorname{Tw}}

\makeatletter
\usepackage[normalem]{ulem}
\usepackage{contour}
\usepackage{mathtools} 

\contourlength{0.8pt} 

\makeatletter
\DeclareRobustCommand{\myuline}[2][0pt]{%
	\ifmmode
	\uline{\hphantom{#2}\kern-#1}%
	\kern#1%
	\mathllap{\mathpalette\my@cont@{#2}}%
	\else
	\uline{\phantom{#2}\kern-#1}%
	\kern#1%
	\llap{\contour{white}{#2}}%
	\fi
}
\newcommand{\my@cont@}[2]{\contour{white}{\mbox{$\m@th#1#2$}}}
\makeatother

\newcommand{\ul}{\myuline}

\usepackage[colorinlistoftodos,prependcaption,textsize=tiny]{todonotes}
\newcommandx{\unsure}[2][1=]{\todo[linecolor=red,backgroundcolor=red!25,bordercolor=red,#1]{#2}}
\newcommandx{\change}[2][1=]{\todo[linecolor=blue,backgroundcolor=blue!25,bordercolor=blue,#1]{#2}}
\newcommandx{\info}[2][1=]{\todo[linecolor=OliveGreen,backgroundcolor=OliveGreen!25,bordercolor=OliveGreen,#1]{#2}}
\newcommandx{\thiswillnotshow}[2][1=]{\todo[disable,#1]{#2}}

\theoremstyle{theorem}  
\newtheorem{mainthm}{Theorem}

\newcommand{\notehelper}[3]{\textcolor{#3}{$\blacksquare$}\marginpar{\ifodd\thepage\raggedright\else\raggedleft\fi\color{#3}\tiny \textbf{#2:} #1}}
\DeclareFontFamily{U}{min}{}
\DeclareFontShape{U}{min}{m}{n}{<-> udmj30}{}

\usepackage{comment}
\usepackage[a4paper, margin=1in]{geometry}

\title{Gros Topoi as partially lax limits of petit topoi}
\author{Fabio Neugebauer, Qi Zhu}
\date{\today}

\begin{document}
\begin{abstract}
    We prove that the sheaf $\infty$-topos associated to a geometric site in the sense of Lurie can be written as a partially lax limit of smaller sheaf $\infty$-topoi. This is thus a formalization of Lurie's vision for fractured $\infty$-topoi to axiomatize the relation between gros and petit topoi. As a consequence, we realize the $\infty$-topos of $\kappa$-small condensed anima as a partially lax limit of sheaf $\infty$-topoi over extremally disconnected spaces, marked at the open embeddings. Moreover, we deduce a gros version of the étale exodromy theorem.
\end{abstract}
\maketitle
\tableofcontents
\section{Introduction}
\noindent Geometric information is often organized in a global or in a local way. One virtue of this principle in topos theory is realized by so-called \emph{gros} or \emph{petit} topoi -- in the example of topological spaces $\Top$ this is the distinction between sheaves on all (small) topological spaces $\Sh(\Top)$ and sheaves $\Sh(X)$ on specific spaces $X$. On the other hand, a global object is ideally an amalgamation of its local information along with suitable gluing information: a sheaf $\mathcal{F} \in \Sh(\Top)$ should be recoverable from a collection of sheaves $\{\mathcal{F}_X \in \Sh(X) \}_{X \in \Top}$ together with certain compatibilities. 

\medskip \noindent Our main theorem is a formalization of this intuition, for which we use Lurie's notion of geometric sites \cite[Section 20]{SAG}. It consists of a site $(\cG, \tau)$ together with a family of admissible morphisms $\cG^{\ad} \subset \cG$ and suitable compatibilities, see \cref{def: geometric site}. In this language, the open embeddings on $\Top$ form a class of admissible morphisms, making our introductory example an instance of this.
\begin{mainthm}[{\Cref{thm:main-theorem}}] \label{mainthm: A}
    Let $(\cG, \cG^{\ad}, \tau)$ be a geometric site and $\cC$ a presentably symmetric monoidal $\infty$-category. There is an equivalence of symmetric monoidal $\infty$-categories
    \[ \Sh(\cG;\cC) \xrightarrow{\; \; \simeq \;\;} \parlaxlim_{X \in \cG^{\op}} \Sh \left(\cG_{/X}^{\ad};\cC \right) \]
    describing the $\infty$-category of $\cC$-valued sheaves on $\cG$ as a partially lax limit indexed by the marked $\infty$-category $(\cG^{\ad})^{\op}\subset  \cG^{\op}$.
\end{mainthm}
The monoidal structure is the sheafification of the pointwise monoidal structure. We work with symmetric monoidal $\infty$-categories but the arguments are equally valid for presentably $\cO$-monoidal categories for any unital $\infty$-operad $\cO$. Moreover, \cref{thm:main-theorem} generalizes to hypersheaf and module categories, cf.~\cref{thm:main-theorem-hyper} and \cref{rem:modules}.

\medskip \noindent Let us informally explain the functoriality of $\Sh(\cG^{\ad}_{/-};\cC)$ defining the partially lax limit. Let us denote by $\cP^{\cC}(-)$ the $\cC$-valued presheaves. Pullback along a map $f\colon X \to Y$ in $\cG$ yields a morphism of sites $f^*\colon \cG^{\ad}_{/Y}\to \cG^{\ad}_{/X}$. Left Kan extension followed by sheafification yields \begin{center}
    \begin{tikzcd}
        \Sh \left(\cG^{\ad}_{/Y};\cC \right) \arrow[r, hookrightarrow] & \cP^{\cC} \left(\cG_{/Y}^{\ad}\right) \arrow[rr, "\Lan_{(f^*)^{\op}}"] & &  \cP^{\cC} \left(\cG_{/X}^{\ad}\right) \arrow[r, "L_{\tau}"] & \Sh \left(\cG_{/X}^{\ad};\cC \right),
    \end{tikzcd}
\end{center}
We will construct this functor rigorously in \cref{section: presheaf as partially lax lim} and \cref{section: sheafifying} with all the desired coherences using $(\infty,2)$-category theory. 
The functor $\Sh(\cG;\cC)\to \Sh(\cG^{\ad}_{/X};\cC)$ appearing on the $X$-factor of \Cref{mainthm: A} is given by restriction along the forgetful functor $\cG^{\ad}_{/X}\to \cG$.

\medskip \noindent Our \cref{mainthm: A} is obtained by sheafifying the following statement about presheaf categories, denoted $\mdef{\cP^{\cC}(-)}\coloneqq \Fun((-)^{\op},\cC)$.
\begin{mainthm}[{\Cref{thm:presheaves-as-lax-lim}}]\label{mainthm: B}
    Let $\cG^{\ad}\subset \cG$ be a pullback-stable left-cancellative wide subcategory of a small $\infty$-category $\cG$. Let $\cC$ be a cocomplete symmetric monoidal $\infty$-category such that the monoidal structure commutes with colimits in each variable. There is a preferred symmetric monoidal equivalence of $\infty$-categories 
    \[  
        \cP^{\cC}(\cG)\isoo  \parlaxlim_{X\in \cG^{\op}} \cP^{\cC}(\cG^{\ad}_{/X}), \quad \Phi\mapsto\left((\cG^{\ad}_{/X})^{\op}\xrightarrow{\fgt} \cG^{\op} \xrightarrow{\Phi} \cC\right)_{X \in \cG^{\op}}
    \]
    where we mark by the morphisms in $\cG^{\ad}$.
\end{mainthm}

\subsection*{Fractured Topoi and Partially Lax Limits} \label{subsec: intro recollections} Let us now informally recall some of the terminology used in \cref{mainthm: A}. We will give a more rigorous recollection in \cref{section: recollection-geometric-sites} and \cref{section: recollection-partially-lax-limits}.

\medskip \noindent Lurie suggested a formalization of the philosophy of gros and petit topoi via so-called \emph{fractured $\infty$-topoi} \cite[Section 20]{SAG}, which roughly is a subcategory $\cX^{\corp}$ of an $\infty$-topos $\cX$ satisfying various compatibilities. While sheaves on sites present $\infty$-topoi, the analogous notion to present fractured $\infty$-topoi is through \emph{geometric sites} -- meaning a site $(\cG, \tau)$ together with a wide subcategory $\cG^{\ad}$ of so-called \emph{admissible morphisms} that is compatible with the topology $\tau$ in a suitable way (\cref{def: geometric site}). By \cite[Theorem 20.6.3.4]{SAG}, the subcategory $\Sh(\cG^{\ad}) \subseteq \Sh(\cG)$ yields a fractured $\infty$-topos, where $\Sh(\cG)$ plays the role of the gros topos and $\{\Sh(\cG^{\ad}_{/X}) \}_{X \in \cG}$ plays the role of the petit topoi. Thus, our main theorem describes the fractured $\infty$-topos $\Sh(\cG)$ as a partially lax limit over the respective petit topoi.

\medskip \noindent Partially lax limits, introduced by Gepner--Haugseng--Nikolaus \cite{GHN2017}, formalize the sense in which one takes a compatible family of objects in the petit topoi. Specializing the rigorous definition from \cref{section: recollection-partially-lax-limits} to our \cref{mainthm: A} describes a sheaf $\cF \in \Sh(\cG;\cC)$ equivalently by the data:
\begin{itemize}
    \item a sheaf $\cF_X \in \Sh \left(\cG^{\ad}_{/X};\cC \right)$ for each $X \in \cG$,
    \item a family of compatible morphisms $f_{\alpha} \colon \alpha^* \cF_Y \to \cF_X$ for every $\alpha \colon X \to Y$ in $\cG$, where $f_{\alpha}$ is an equivalence if $\alpha$ is admissible.
\end{itemize}
Moreover, this equivalence is compatible with the pointwise symmetric monoidal structure on the partially lax limit (\cref{remark: symmetric monoidal lax limits}).

\subsection*{Simplification of the Partially Lax Limit on Certain Functors}
It turns out that there are a number of functors which send partially lax limits to limits -- we will recall some in \Cref{subsec:functors-that-send-partially-lax-limits-to-limits}. We apply this in the situation of \Cref{mainthm: A} and gain extra mileage when $\cG$ admits a terminal object $\ast$ so that $\lim_{\cG^{\op}}$ is computed by evaluation at $\ast$. In this case, the restriction functor 
\begin{align}
    \Sh(\cG;\cC)\too \Sh \left(\cG^{\ad}_{/\ast}; \cC \right) \label{eq:restriction-to-terminal-petit-site}
\end{align}
induces an equivalence on
\begin{enumerate}
    \item the full subcategories of dualizable objects,
    \item Picard spectra, Picard spaces and Picard groups, and
    \item rigidifications in the sense of Gaitsgory--Rozenblyum \cite{GS17, HSSS21, ramzi2026locallyrigidinftycategories},
\end{enumerate}
see \cref{corollary: R of Sh}. While these applications admit a more elementary proof via the left adjoint of \cref{eq:restriction-to-terminal-petit-site}, the partially lax limit perspective motivated us to expect these results in the first place.

\subsection*{Condensed Objects as a Partially Lax Limit}
We want to single out an example of fractured $\infty$-topoi from condensed mathematics, while more classical examples are recalled in \Cref{section: recollection-geometric-sites}.

\medskip \noindent Fix an uncountable strong limit cardinal $\kappa$. Let $\ExtrDisc$ denote the $1$-category of $\kappa$-small extremally disconnected compact Hausdorff spaces, i.e.~those $\kappa$-small compact Hausdorff spaces in which the closure of an open set is open again.\footnote{Let us recall that every extremally disconnected compact Hausdorff space is a profinite set.} Equivalently, these are the projective objects in the category of compact Hausdorff spaces $\CHaus$ by a theorem of Gleason \cite[Theorem 2.5]{gleason58}. We endow it with the Grothendieck topology of finitely jointly surjective families of maps. Let $\cC$ be a cocomplete $\infty$-category, then Clausen--Scholze resp.~Barwick--Haine \cite{barwick2019pyknoticobjectsibasic, scholze2026lecturescondensedmathematics, scholze2026lecturesanalyticgeometry} introduced the \emph{condensed objects valued in $\cC$} as the $\infty$-category of $\cC$-valued sheaves on $\ExtrDisc$, i.e.
\[ \Cond(\cC) \coloneqq \Sh(\ExtrDisc; \cC) \simeq \Fun^{\times}(\ExtrDisc^{\op};\cC). \] In work of the second-named author with Rasekh, we identified $\Cond(\An)$ as a fractured $\infty$-topos presented by a geometric site and computed the associated petit topoi \cite[Theorems A and B]{rasekhzhu2026fracturedstructurescondensedmathematics}. Combining this with \cref{mainthm: A} we deduce:

\begin{mainthm}[\cref{corollary: Cond as lax limit}]
    Let $\cC$ be a presentably symmetric monoidal $\infty$-category. Then, there is an equivalence of symmetric monoidal $\infty$-categories
    \[ \Cond(\cC) \simeq \parlaxlim_{X \in \ExtrDisc^{\op}} \Sh(X;\cC), \]
    where the open embeddings are marked.
\end{mainthm}

\noindent The appeal of Theorem~C is that it identifies  $\Cond(\Sp)$ as a partially lax limit of \enquote{extremely\footnote{...maybe we should say extremally!} nice} categories. Indeed, for any $X\in \ExtrDisc$ the symmetric monoidal category $\Sh(X;\Sp)$ is compactly generated \cite[Proposition 3.1]{harr25} and rigid \cite[Proposition 2.20]{HSSS21}, i.e.~dualizable, objects of $\Sh(X;\Sp)$ are precisely the compact objects. Moreover, equivalences in $\Sh(X;\Sp)$ are detected on the points of $X$.

\subsection*{Gros Étale Exodromy}

Van Dobben de Bruyn, building on work of Barwick--Glasman--Haine and Wolf \cite{barwick2020exodromy, wolf2022}, proves an exodromy correspondence for Postnikov complete étale sheaves: This is a way to reconstruct sheaves on a scheme as continuous functors (see \cref{section: gros exodromy}) from a so-called Galois category \cite[Theorem 2]{debruyn2026condensedproofproetaleetale}. 

\medskip \noindent Here, the \emph{condensed Galois $\infty$-category} of $X$ from \cite{barwick2020exodromy, haine2025condensedhomotopytypescheme, mair2026galoiscategoriescondensedcontractible, debruyn2026condensedproofproetaleetale} is
\[ \Gal(X) \colon \ExtrDisc^{\op} \longrightarrow \Cat_{\infty}, \ S \mapsto \Fun^*_{\mathrm{loc} \, \mathrm{coh}}(\Sh(X_{\et}), \Sh(S)), \]
which for $S \in \ExtrDisc^{\op}$ consists of locally coherent geometric morphisms $\Sh(S) \to \Sh(X_{\et})$, see \cite[2.1.6]{debruyn2026condensedproofproetaleetale}. This is functorial in $X$ and our theorem concerns the gros version
$$ \GAL \colon  \ExtrDisc^{\op} \longrightarrow \Cat_{\infty}, \qquad S\mapsto  \laxcolimdag_{X \in \Sch} \Gal(X)(S),$$
where we mark the étale maps and run over $\kappa$-small qcqs schemes, see \cref{def: gros Galois category}. Moreover, let $\ul{\cC} = \Sh(-;\cC) \colon \ExtrDisc^{\op} \to \Cat_{\infty}$ be the condensed $\infty$-category associated to $\cC$. Then, we prove a gros version of the étale exodromy theorem:

\begin{mainthm}[\cref{theorem: gros etale exodromy}]
    Let $\cC$ be a compactly assembled presentable $\infty$-category. Then, there is a preferred étale exodromy equivalence
    \[ \Sh^{\post}_{\et}(\Sch; \cC) \xrightarrow{\ \simeq \ } \Fun^{\cts}(\GAL, \ul{\cC}) \]
    of $\infty$-categories.
\end{mainthm}

\noindent The same result holds over more general base schemes $T$, see \cref{remark: exodromy base scheme}.

\subsection*{Conventions} We will freely use the language of $(\infty,2)$-categories, for which we recommend \cite[Appendix A]{blansblom} as a reference.

\subsection*{AI Disclosure} We have used LLMs to search the literature, help proofread and improve parts of the language. The mathematical content is entirely our own.

\subsection*{Acknowledgements} We thank Thomas Blom and Phil Pützstück for insightful conversations about partially lax limits, Remy van Dobben de Bruyn for answering questions about exodromy and Markus Hausmann for giving us feedback on a draft. FN is supported by 
the European Research Council (ERC) under Horizon Europe (Starting Grant 
BorSym, ID:~101163408) and furthermore thanks the Max Planck Institute for Mathematics in Bonn for its hospitality. QZ is grateful to the Max Planck Institute for Mathematics in Bonn for its hospitality and financial support.
\newpage

\section{Recollection on Geometric Sites}\label{section: recollection-geometric-sites}
Lurie's fractured $\infty$-topoi \cite[Section 20]{SAG} are usually presented by a geometric site. For the purpose of this article, we will only need the notion of geometric sites. Let us recall Lurie's definition.

\begin{definition}[{\cite[Definition 20.2.1.1, 20.6.2.1]{SAG}}] \label{def: geometric site}
Let $\cG$ be a small $\infty$-category.
\begin{enumerate}
    \item A \tdef{class of admissible morphisms} is a class of morphisms in $\cG$ that is a full and replete subcategory $\mdef{\Ar^{\ad}(\cG)} \subseteq \Ar(\cG)$ of the arrow category such that:
    \begin{enumerate}[label=\alph*)]
    \item For any object $X\in \cG$ the identity morphism $\id_X$ is admissible.
    \item For any admissible morphism $a \colon A\to B$ and any morphism $f \colon X\to B$ in $\cG$ the pullback $a'\colon A\times_B X\to X$ exists and is admissible.
    \item Let $f,g$ be composable morphisms in $\cG$. If $g$ is admissible, then $g\circ f$ is admissible if and only if $f$ is admissible. 
    \item \label{item:retract-of-admissilbe} The subcategory $\Ar^{\ad}(\cG) \subseteq \Ar(\cG)$ is closed under retracts.
    \end{enumerate} 
    We denote by $\mdef{\cG^{\ad}} \subset \cG$ the wide subcategory spanned by the admissibles.
    \item A \tdef{geometric site} is a site $(\cG, \tau)$ together with a class of admissible morphisms such that any $\tau$-cover admits a $\tau$-subcover generated by a set of admissible morphisms.
\end{enumerate}
\end{definition}

\begin{remark}\label{rmk:retract-unneccesary}
    In this article we will \emph{never} need condition \ref{item:retract-of-admissilbe}. It is only relevant to obtain the full strength of Lurie's fractured $\infty$-topoi from \cite[Section 20]{SAG}, which will play no further role in this article.
\end{remark}

\noindent One can pull back the topology $\tau$ to $\tau^{\ad}$ on $\cG^{\ad}$ by \cite[Proposition 20.6.1.1]{SAG}. Then, $(\Sh(\cG), \Sh(\cG^{\ad}))$ forms a fractured $\infty$-topos in the sense of Lurie \cite[Theorem 20.6.3.4]{SAG}. This formalizes the notion of gros and petit topoi by letting $\Sh(\cG)$ take the role of a gros topos and the slices $\Sh(\cG^{\ad}_{/X})$ for $X \in \cG$ take the role of petit topoi.

\medskip \noindent To deal with size issues, fix an uncountable strong limit cardinal $\kappa$.
  
\begin{example} \label{example: alggeo geometric site}
    The $\infty$-category of qcqs $\kappa$-small\footnote{In the sense of \cite[Definition 1.1.7]{debruyn2026condensedproofproetaleetale}. The precise definition won't be important to us, but a size bound will be needed in the gros exodromy theorem (\cref{theorem: gros etale exodromy}), since \cite[Theorem 3.2.6]{debruyn2026condensedproofproetaleetale} requires such a size condition.} schemes or spectral schemes with the étale/Zariski topology together with the étale morphisms/open immersions forms a geometric site \cite[20.6.4]{SAG}.
\end{example}

\begin{example} \label{example: geometry geometric site}
    Consider the (1-)category $\Mfld^r$ of $r$-times differentiable (second countable, Hausdorff) $\kappa$-small manifolds with $r$-times differentiable maps together with a topology given by jointly surjective open embeddings. Together with the $r$-times differentiable open embeddings $\Mfld_{\text{ét}}^r$, we obtain a geometric site \cite[Lemma 2.4]{clough25}. Let $\cC$ be an $\infty$-category. The slices can be computed as $\Sh(\Mfld_{\text{ét}/M}^r;\cC) \simeq \Sh(M;\cC)$ for $M \in \Mfld^r$, see \cite[Remark 2.6]{clough25}.
\end{example}
  
\begin{example} \label{example: condensed geometric site}
    Let $\ExtrDisc$ denote the (1-)category of $\kappa$-small\footnote{Alternatively, one may take the pyknotic perspective from \cite{barwick2019pyknoticobjectsibasic} to deal with the set theory.} extremally disconnected compact Hausdorff spaces, i.e.~those  $\kappa$-small compact Hausdorff spaces in which the closure of an open set is open again. Consider a Grothendieck topology $\tau$ given by the finitely jointly surjective families of maps. Let $\cC$ be an $\infty$-category. Then, Clausen--Scholze's $\infty$-category of \tdef{condensed objects with values in $\cC$} is
    \[ \mdef{\Cond(\cC)} \coloneqq \Sh(\ExtrDisc; \cC) \simeq \Fun^{\times}(\ExtrDisc^{\op},\cC), \]
    see \cite[Definition 11.7]{scholze2026lecturesanalyticgeometry}. For example, a \tdef{condensed $\infty$-category} is a condensed object valued in $\Cat_{\infty}$.
    
    \medskip \noindent Let $\ExtrDisc^{\open}$ denote the wide subcategory of open embeddings. Then, $(\ExtrDisc, \ExtrDisc^{\open}, \tau)$ defines a geometric site by \cite[Theorem A]{rasekhzhu2026fracturedstructurescondensedmathematics}. Furthermore, \cite[Theorem B]{rasekhzhu2026fracturedstructurescondensedmathematics} computes the slices as $\Sh(\ExtrDisc^{\open}_{/S};\cC) \simeq \Sh(S;\cC)$ for $S \in \ExtrDisc$.
    
    \medskip \noindent The main result (\cref{mainthm: A}) of this article relates $\Cond(\cC)$ and $\Sh(S;\cC)$ via partially lax limits.
\end{example}

\section{On Partially Lax Limits}\label{section: recollection-partially-lax-limits}
\subsection{Recollection} 
Let us now give a formal version of the informal recollection of partially lax limits from \Cref{subsec: intro recollections}. We recommend \cite[Section 2]{GLP} for further background. We denote by $\Cat_{\infty}$ the $\infty$-category\footnote{To apply the lax natural transformation machine, we view $\Cat_{\infty}$ as an $(\infty,2)$-category via the self-enrichment coming from the cartesian closedness.} of small $\infty$-categories, but the analogous constructions hold for large $\infty$-categories by passing to a larger universe.

\medskip \noindent Let $F\colon \cI\to \Cat_{\infty}$ be a functor and let $\ast$ denote the terminal $\infty$-category. Informally, a point in its lax limit consists of functors $F_i\colon *\to F(i)$ for every $i\in\cI$, together with natural transformations
\[
    \begin{tikzcd}[column sep=large, row sep=large]
        * \arrow[r, "\cF_i"] \arrow[dr, "\cF_j"', ""{name=cone, above}]
        & F(i) \arrow[d, "F(\alpha)"] \\
        & F(j)
        \arrow[Rightarrow, from=1-2, to=cone,
            "\eta_{\alpha}" description,
            shorten <=8pt, shorten >=3pt]
    \end{tikzcd}
    \qquad
    \eta_{\alpha}\colon F(\alpha)\circ F_i\Longrightarrow F_j
\]
for every morphism $\alpha\colon i\to j$ in $\cI$, satisfying certain coherences. Equivalently, $\cF_i$ is an object of $F(i)$ and $\eta_{\alpha}\colon F(\alpha)(\cF_i)\to \cF_j$ gives its transition morphism in $F(j)$. These data form a lax cone over $F$, i.e.~a lax natural transformation $\Delta(\pt)\xRightarrow{\text{lax}}F$ from the constant functor $\Delta(\pt)$ at the terminal category $\pt$ to $F$. Reversing the natural transformations gives an oplax cone.

\begin{definition}\label{defn:lax-limit}
    The \tdef{lax limit} and \tdef{oplax limit} of $F\colon\cI\to\Cat_{\infty}$ are the $\infty$-categories of lax resp.~oplax cones:
    \[
        \mdef{\laxlim_{\cI}F}\coloneqq\Nat^{\lax}(\Delta(*),F)
        \quad\text{and}\quad
        \mdef{\oplaxlim_{\cI}F}\coloneqq\Nat^{\oplax}(\Delta(*),F),
    \]
    where $\Delta(*)$ denotes the constant functor with value the terminal $\infty$-category. Similarly, the \tdef{(op-)lax colimits} are given by \tdef{(op-)lax cocones}.
\end{definition}

\noindent By straightening/unstraightening for lax natural transformations, these descriptions identify with categories of sections, see \cite[Remark 2.7]{GLP}. In particular,
\[
    \laxlim_{\cI}F\simeq\Fun_{/\cI}(\cI,\Un^{\co}F),
\]
where $\Un^{\co}F\to\cI$ is the coCartesian unstraightening of $F$. Under this equivalence, $\eta_{\alpha}$ is an equivalence precisely when the corresponding section sends $\alpha$ to a coCartesian edge. Similarly, the oplax limit is the $\infty$-category of sections of the Cartesian unstraightening $\Un^{\ct}F\to\cI^{\op}$.

\begin{definition}
    A \tdef{marked $\infty$-category} is a pair $\cI^{\dagger}=(\cI,\cW)$ consisting of an $\infty$-category $\cI$ together with a replete subcategory $\cW\subseteq\cI$ containing all objects. For two marked $\infty$-categories $\cC,\cD$, we denote by $\mdef{\Fun^{\dagger}(\cC,\cD)}\subseteq\Fun(\cC,\cD)$ the full subcategory of functors preserving marked edges.

    \medskip \noindent The \tdef{partially (op-)lax limits} are the full subcategories
    \[
        \mdef{\parlaxlim_{\cI^{\dagger}}F}\subseteq\laxlim_{\cI}F
        \quad\text{and}\quad
        \mdef{\oplaxlimdag_{\cI^{\dagger}}F}\subseteq\oplaxlim_{\cI}F
    \]
    spanned by those (op-)lax cones whose structure transformations $\eta_{\alpha}$ are equivalences for all $\alpha\in\cW$, i.e.~whose restriction to $\cW$ is strictly natural. Similarly, the \tdef{partially (op-)lax colimit} is characterized by the universal (op-)lax cocone whose restriction to $\cW$ is strictly natural.
\end{definition}

\begin{example}\label{exmp:limit-in-lax-limit}
    The limit of $F$ is the full subcategory
    \[
        \lim_{\cI}F\;\subseteq\;\parlaxlim_{\cI^{\dagger}}F
        \;\subseteq\;\laxlim_{\cI}F
    \]
    spanned by those lax cones for which $\eta_{\alpha}$ is an equivalence for every morphism $\alpha$ in $\cI$. Equivalently, it consists of the coCartesian sections of $\Un^{\co}F\to\cI$. Thus, marking all morphisms recovers the limit, while marking only equivalences recovers the lax limit. Partially lax limits interpolate between these two constructions.
\end{example}

\begin{remark}\label{remark: berman theorem}
    Let $\cI^{\dagger}$ be a marked $\infty$-category and $F\colon\cI\to\Cat_{\infty}$ a functor. We write $|-|$ for the localization at the marked arrows and give the slice categories the markings induced by their forgetful functors to $\cI$. The end and coend formulas of \cite[Definition 3.3 and Remark 3.4]{berman2020laxlimitsinfinitycategories} with the conventions of \cite{LNP, GLP} give
    \begin{align*}
        \parlaxlim_{\cI^{\dagger}}F
        &\simeq\lim_{(i\to j)\in\Tw(\cI)^{\op}}
        \Fun\left(|\cI^{\dagger}_{/i}|,F(j)\right),\\
        \laxcolimdag_{\cI^{\dagger}}F
        &\simeq \underset{(i\to j)\in\Tw(\cI)}{\colim}\,
        |(\cI^{\op})^{\dagger}_{/j}|\times F(i).
    \end{align*}
    Here we use the twisted arrow category with projection $\Tw(\cI)\to\cI\times\cI^{\op}$, $(i\to j)\mapsto(i,j)$. The formulas for the op-versions are obtained by replacing the localized slice categories by their opposites.

    \medskip \noindent We mark $\Un^{\co}F\to\cI$ and $\Un^{\ct}F\to\cI^{\op}$ by the coCartesian resp.~Cartesian lifts of marked edges. By \cite[Theorem 4.4]{berman2020laxlimitsinfinitycategories},
    \begin{alignat*}{3}
    \laxcolimdag_{\cI^{\dagger}} F
    &\simeq \left| \Un^{\ct} F \right|
    &\qquad \text{and} \qquad &&
    \oplaxcolimdag_{\cI^{\dagger}} F
    &\simeq \left| \Un^{\co} F \right|
    \\
    \laxlimdag_{\cI^{\dagger}} F
    &\simeq \Fun_{/\cI^{\dagger}}^{\dagger}
        \left(\cI^{\dagger}, \Un^{\co} F\right)
    &\qquad \text{and} \qquad &&
    \oplaxlimdag_{\cI^{\dagger}} F
    &\simeq \Fun_{/\cI^{\dagger \op}}^{\dagger}
        \left(\cI^{\dagger \op}, \Un^{\ct} F\right).
    \end{alignat*}
    In particular, the partially lax limit consists of those sections that send marked edges to coCartesian edges, also explaining \cref{exmp:limit-in-lax-limit}.
\end{remark}

\begin{example}
    Let $(\cG,\cG^{\ad})$ be an admissibility structure. Then, $(\cG,\cG^{\ad})$ is a marked $\infty$-category. Our \Cref{mainthm: A} concerns a partially lax limit over its opposite.
\end{example}

\begin{remark}[{\cite[Proposition 2.9]{GLP}}]\label{remark: symmetric monoidal lax limits}
    Let $\CMon(\Cat_{\infty})$ denote the $\infty$-category of symmetric monoidal $\infty$-categories and symmetric monoidal functors. Given a functor $F \colon \cI^{\dagger} \to \CMon(\Cat_{\infty})$ its partially lax limit $\parlaxlim_{\cI^{\dagger}}F$ admits a preferred symmetric monoidal structure, computed pointwise. It is characterized by the natural equivalence
    \[
        \Fun^{\otimes\text{-}\lax}
        \left(\cC,\parlaxlim_{\cI^{\dagger}}F\right)
        \simeq
        \parlaxlim_{\cI^{\dagger}}
        \Fun^{\otimes\text{-}\lax}(\cC,F(-))
    \]
    for every symmetric monoidal $\infty$-category $\cC$, where $\Fun^{\otimes\text{-}\lax}$ denotes the $\infty$-category of lax symmetric monoidal functors and monoidal natural transformations. Thus, a lax symmetric monoidal functor into the partially lax limit is the same datum as a partially lax cone of lax symmetric monoidal functors. It is symmetric monoidal if and only if all its components are symmetric monoidal. The same statements hold for (partially) oplax limits.
\end{remark}

\noindent We end this recollection with the following symmetric monoidal version of \cite[Proposition 2.21]{GLP}.
\begin{lemma}\label{prop:oplax-vs-lax-lim-for-adjoints-diagram}
    Let $G\colon\cI\to\Cat_{\infty}$ be a diagram of right adjoints and let $F\colon\cI^{\op}\to\Cat_{\infty}$ be its associated diagram of left adjoints. Then, there exists an equivalence
    \[
        \oplaxlim_{\cI}G\simeq\laxlim_{\cI^{\op}}F.
    \]
    If, moreover, $G$ is a diagram of symmetric monoidal right adjoints and the induced oplax symmetric monoidal structures on the left adjoints are strong, then this equivalence is symmetric monoidal for the preferred symmetric monoidal structures.
\end{lemma}
\begin{proof}
    As in \cite[Proposition 2.21]{GLP}, taking mates gives an equivalence of $\infty$-categories
    \[
        \oplaxlim_{\cI}G
        =\Nat^{\oplax}(\Delta(*),G)
        \simeq\Nat^{\lax}(\Delta(*),F)
        =\laxlim_{\cI^{\op}}F.
    \]
    For the symmetric monoidal statement, the adjunctions lift to symmetric monoidal categories and lax symmetric monoidal functors, see \cite[Corollary C]{monoidal_mate_equiv}. For every symmetric monoidal $\infty$-category $\cC$, postcomposition therefore gives adjoint diagrams $\Fun^{\otimes\text{-}\lax}(\cC,F(-))$ and $\Fun^{\otimes\text{-}\lax}(\cC,G(-))$. Applying the first statement to these diagrams and using \Cref{remark: symmetric monoidal lax limits} gives a natural equivalence
    \[
        \Fun^{\otimes\text{-}\lax}
        \left(\cC,\oplaxlim_{\cI}G\right)
        \simeq
        \Fun^{\otimes\text{-}\lax}
        \left(\cC,\laxlim_{\cI^{\op}}F\right).
    \]
    The claim follows by Yoneda.
\end{proof}
\subsection{Functors that send partially lax limits to limits}\label{subsec:functors-that-send-partially-lax-limits-to-limits}
Throughout this subsection, we denote by $\mathbb{C}$ either the $\infty$-category of presentably symmetric monoidal $\infty$-categories $\CAlg(\PrL)$ or its full subcategory $\CAlg(\PrLst)$ spanned by stable $\infty$-categories. We consider functors $\cR\colon\mathbb{C}\to\cD$, where $\cD$ is any $\infty$-category. We think of $\cR$ as extracting an invariant of presentably symmetric monoidal categories, e.g.~the Picard group.

\medskip \noindent For a small marked $\infty$-category $\cI^{\dagger}=(\cI,\cW)$ and a functor $F\colon\cI\to\mathbb{C}$, we equip $\parlaxlim_{\cI^{\dagger}}F$ with the preferred symmetric monoidal structure from \Cref{remark: symmetric monoidal lax limits}. The partially lax limit is presentable, with colimits computed pointwise, and its tensor product preserves colimits in each variable. It is stable if the categories $F(i)$ are stable. We therefore regard $\parlaxlim_{\cI^{\dagger}}F$ as an object of $\mathbb{C}$. The inclusion $\lim_{\cI}F\to\parlaxlim_{\cI^{\dagger}}F$ from \Cref{exmp:limit-in-lax-limit} is symmetric monoidal and preserves colimits, so it too is a morphism in $\mathbb{C}$.

\begin{definition}\label{def: strictifying functor}
    We call $\cR$ \tdef{lax-invariant} if, for every small marked $\infty$-category $\cI^{\dagger}=(\cI,\cW)$ and every functor $F\colon\cI\to\mathbb{C}$, the map
    \begin{equation}\label{eq: strictifying comparison}
        \cR\left(\lim_{\cI}F\right)
        \too
        \cR\left(\parlaxlim_{\cI^{\dagger}}F\right)
    \end{equation}
    induced by \Cref{exmp:limit-in-lax-limit} is an equivalence.
\end{definition}

\begin{proposition}\label{prop: criterion for strictifying functors}
    Suppose that $\cR$ preserves small limits. Then $\cR$ is lax-invariant if and only if, for every $\cC\in\mathbb{C}$, the constant-functor map induces an equivalence
    \[
        \cR(\cC)\too\cR\left(\Fun(\Delta^1,\cC)\right),
    \]
    where the functor category is equipped with its pointwise symmetric monoidal structure.
\end{proposition}
\begin{proof}
    Necessity follows from
    \[
        \cC \simeq \lim_{\Delta^1}\Delta(\cC) \qquad \text{and}
        \qquad
        \Fun(\Delta^1,\cC) \simeq \laxlim_{\Delta^1}\Delta(\cC).
    \]
    Conversely, suppose that the displayed map is an equivalence for every $\cC$. For a small $\infty$-category $\cK$, let $\cK^{\mathrm{gpd}}$ denote its groupoid completion. Restriction induces a natural map
    \[
        \cR\left(\Fun(\cK^{\mathrm{gpd}},\cC)\right)
        \too
        \cR\left(\Fun(\cK,\cC)\right).
    \]
    Both sides send colimits in $\cK$ to limits: groupoid completion preserves colimits, functor categories turn colimits in their first variable into limits in $\mathbb{C}$, and $\cR$ preserves limits. The map is an equivalence for $\cK=\Delta^1$ by assumption. Consequently, it is an equivalence for every $\cK$, since $\Cat_{\infty}$ is generated under small colimits by $\Delta^1$. Indeed, $\Cat_{\infty}$ is generated by the simplices $\Delta^n$ by \cite[Example 20.4.1.9]{SAG}; moreover, $\Delta^0$ is a retract of $\Delta^1$, and for $n\geq 1$ we have $\Delta^n \simeq \Delta^1 \amalg_{\Delta^0} \cdots \amalg_{\Delta^0} \Delta^1$ in $\Cat_{\infty}$.

    \medskip \noindent In particular, if $\cK^{\mathrm{gpd}}$ is contractible, the constant-functor map 
    \[
        \cR(\cC)\longrightarrow \cR\left(\Fun(\cK^{\mathrm{gpd}}, \cC) \right)\longrightarrow \cR\left(\Fun(\cK,\cC)\right)
    \]
    is an equivalence.
    
    \medskip \noindent This applies to $\cK=|\cI^{\dagger}_{/i}|$: localization does not change groupoid completion, and $\cI_{/i}$ has a terminal object. The end formula from \Cref{remark: berman theorem}, with pointwise symmetric monoidal structures, therefore gives
    \begin{align*}
        \cR\left(\parlaxlim_{\cI^{\dagger}}F\right)
        &\simeq\lim_{(i\to j)\in\Tw(\cI)^{\op}}
        \cR\left(\Fun\left(|\cI^{\dagger}_{/i}|,F(j)\right)\right)\\
        &\simeq\lim_{(i\to j)\in\Tw(\cI)^{\op}}\cR(F(j))
        \simeq\lim_{\cI}(\cR\circ F)
        \simeq\cR\left(\lim_{\cI}F\right).
    \end{align*}
    Here the third equivalence uses that the target projection $\Tw(\cI)^{\op}\to\cI$ is limit-cofinal. These equivalences identify \eqref{eq: strictifying comparison} with an equivalence.
\end{proof}

\begin{proposition}\label{prop: dualizable objects of partially lax limits}
    The functor
    \[
        (-)^{\dbl}\colon\CAlg(\PrL) \longrightarrow \Catl_{\infty}
    \]
    taking the full subcategory of dualizable objects preserves small limits and is lax-invariant. In particular, there is a preferred symmetric monoidal equivalence
    \[
        \lim_{i\in\cI}F(i)^{\dbl}
        \isoo
        \left(\parlaxlim_{\cI^{\dagger}}F\right)^{\dbl}.
    \]
\end{proposition}
\begin{proof}
    In a limit of symmetric monoidal $\infty$-categories, dualizability is detected componentwise. Thus, $(-)^{\dbl}$ preserves small limits.

    \medskip \noindent By \Cref{prop: criterion for strictifying functors}, it remains to show that
    \[
        \cC^{\dbl}\too\Fun(\Delta^1,\cC)^{\dbl}
    \]
    is an equivalence. We claim that an arrow $f\colon X\to Y$ is dualizable for the pointwise tensor product precisely when $X,Y$ are dualizable and $f$ is an equivalence.

    \medskip \noindent Suppose that $f$ is dualizable. Evaluation at the endpoints shows that $X$ and $Y$ are dualizable, and we may write its dual as $g\colon\bD(X)\to\bD(Y)$. Let
    \[
        \ev_X\colon\bD(X)\otimes X\to\mathbf{1},
        \qquad
        \mathrm{coev}_X\colon\mathbf{1}\to X\otimes\bD(X)
    \]
    denote the duality maps, and similarly for $Y$. Naturality of evaluation and coevaluation in the arrow category gives
    \[
        \ev_Y\circ(g\otimes f)\simeq\ev_X,
        \qquad
        (f\otimes g)\circ\mathrm{coev}_X
        \simeq\mathrm{coev}_Y.
    \]
    We claim that $g$ and $\bD(f)\colon\bD(Y)\to\bD(X)$ are mutually inverse. By the defining identities for $\bD(f)$,
    \begin{align*}
        \ev_X\circ\bigl((\bD(f)\circ g)\otimes\id_X\bigr)
        &\simeq\ev_Y\circ(g\otimes f)
        \simeq\ev_X,\\
        \bigl(\id_Y\otimes(g\circ\bD(f))\bigr)
        \circ\mathrm{coev}_Y
        &\simeq(f\otimes g)\circ\mathrm{coev}_X
        \simeq\mathrm{coev}_Y.
    \end{align*}
    The triangle identities imply that evaluation and coevaluation induce equivalences
    \begin{align*}
        \Map(\bD(X),\bD(X))
        &\simeq\Map(\bD(X)\otimes X,\mathbf{1}),\\
        \Map(\bD(Y),\bD(Y))
        &\simeq\Map(\mathbf{1},Y\otimes\bD(Y)).
    \end{align*}
    Hence $\bD(f)\circ g\simeq\id_{\bD(X)}$ and $g\circ\bD(f)\simeq\id_{\bD(Y)}$. Thus $\bD(f)$, and therefore $f$, is an equivalence.

    \medskip \noindent Conversely, an equivalence $f \colon X \to Y$ between dualizable objects is equivalent to $\id_X$ in $\Fun(\Delta^1, \cC)$, which is dualizable since $X$ is.
\end{proof}
\begin{corollary}
     The Picard spectrum, the Picard space and the Picard group functors are lax-invariant. 
\end{corollary}
\begin{proof}
    These functors factor through the lax-invariant functor $(-)^{\dbl}$ from \Cref{prop: dualizable objects of partially lax limits}.
\end{proof}
\begin{proposition}\label{prop: rigidification of partially lax limits}
    Let $\CAlg^{\mathrm{rig}}\subseteq\CAlg(\PrLst)$ denote the full subcategory of rigid $\infty$-categories in the sense of Gaitsgory--Rozenblyum \cite{GS17, HSSS21, ramzi2026locallyrigidinftycategories}. Rigidification
    \[
        \Rig\colon\CAlg(\PrLst) \longrightarrow \CAlg^{\mathrm{rig}}
    \]
    preserves small limits and is lax-invariant. Thus, for every small marked $\infty$-category $\cI^{\dagger}$ and every $F\colon\cI\to\CAlg(\PrLst)$, the canonical map
    \[
        \Rig\left(\lim_{\cI}F\right)
        \too\Rig\left(\parlaxlim_{\cI^{\dagger}}F\right)
    \]
    is an equivalence. 
\end{proposition}
\begin{proof}
    Rigidification is right adjoint to the inclusion of rigid categories  \cite[Theorem 4.81]{ramzi2026locallyrigidinftycategories}, and hence preserves limits with values in $\CAlg^{\mathrm{rig}}$. By \Cref{prop: criterion for strictifying functors}, it suffices to verify the constant-arrow condition.

    \medskip \noindent Let $\cA$ be rigid and consider the map 
    \[ \Map_{\CAlg(\Pr_{\mathrm{st}}^L)}(\cA, \cC) \longrightarrow \Map_{\CAlg(\Pr_{\mathrm{st}}^L)} \left(\cA, \Fun(\Delta^1, \cC) \right) \]
    induced by the constant-functor map. A symmetric monoidal left adjoint $\cA\to\Fun(\Delta^1,\cC)$ is equivalently a monoidal natural transformation between two symmetric monoidal left adjoints $\cA\to\cC$. Such a transformation is an equivalence by \cite[Lemma 4.94]{ramzi2026locallyrigidinftycategories}. Thus, the map on mapping spaces is an equivalence, so the universal property of rigidification and Yoneda give
    \[
        \Rig(\cC)\isoo\Rig\left(\Fun(\Delta^1,\cC)\right),
    \]
    as required. This is the same argument underlying the lax pullback result \cite[Corollary 4.95]{ramzi2026locallyrigidinftycategories}.
\end{proof}

\begin{remark}
    The target in \Cref{prop: rigidification of partially lax limits} matters: the limit of rigidifications is taken in $\CAlg^{\mathrm{rig}}$. The forgetful functor to $\Catl_{\infty}$ does not generally preserve these limits, see \cite[Example 4.88]{ramzi2026locallyrigidinftycategories}.
\end{remark}

\section{Presheaf Categories as Partially Lax Limits} \label{section: presheaf as partially lax lim}
The goal of this section is to identify presheaf categories of admissibility structures with a partially lax limit, i.e.~to prove \cref{mainthm: B}.
\subsection*{The Lax Limit}

For this section we consider a small $\infty$-category $\cG$ together with an admissibility structure $\cG^{\ad}$ in the sense of \cref{def: geometric site}. Moreover, fix a cocompletely symmetric monoidal $\infty$-category $\cC$, i.e.~it has all colimits and the tensor product commutes with colimits in each variable.

\medskip \noindent Since the coCartesian unstraightening of $\cG^{\op} \to \Cat_{\infty}, \ X\mapsto (\cG_{/X}^{\ad})^{\op}$ is the opposite of the target functor $t\colon \Ar^{\ad}(\cG)\to \cG$, the oplax colimit of $X\mapsto (\cG_{/X}^{\ad})^{\op} $ is
\[
     \oplaxcolim_{X\in \cG^{\op}}\;  (\cG_{/X}^{\ad})^{\op} \isoo \Ar^{\ad}(\cG)^{\op},
\]
by \cref{remark: berman theorem}; with initial oplax cocone
    \[
        (\cG^{\ad}_{/-})^{\op} \xRightarrow{\text{oplax}}  \const(\Ar^{\ad}(\cG)^{\op}) , \qquad \left\{\fgt_X^{\op}\colon (\cG^{\ad}_{/X})^{\op} \to  \Ar^{\ad}(\cG)^{\op} \right\}_{X\in \cG^{\op}}.
    \]

For our fixed cocompletely symmetric monoidal $\infty$-category $\cC$ there is an $(\infty, 2)$-functor 
\begin{align}
    \Fun(-,\cC) \colon \tCat^{\text{$1$-op}} \too \CatRotimes, \qquad \cD \mapsto \Fun(\cD, \cC) \label{eq:presheaf-infinity-two-functor}
\end{align}
which equips functor categories with the pointwise symmetric monoidal structure and precomposition functoriality. Here $\CatRotimes$ is the $(\infty,2)$-category of large symmetric monoidal $\infty$-categories and symmetric monoidal right adjoint functors, and $(-)^{\text{$1$-op}}$ means that the $1$-morphisms are reversed. For these constructions and the language of $(\infty,2)$-categories we refer to \cite[Section 2 and Appendix A]{blansblom}.

\medskip \noindent The functor from \Cref{eq:presheaf-infinity-two-functor} sends oplax colimits to oplax limits by the $2$-op version of \cite[Proposition 4.11]{LNP}. Thus, we have constructed a symmetric monoidal equivalence:
    \begin{align}\label{eq:presheaves-on-arrows-oplax}
        \cP^{\cC}(\Ar^{\ad}(\cG)) \isoo \oplaxlim_{X\in \cG} \cP^{\cC}(\cG^{\ad}_{/X}), \quad \Phi \mapsto \left((\cG^{\ad}_{/X})^{\op}\xrightarrow{\fgt} \Ar^{\ad}(\cG)^{\op} \xrightarrow{\Phi} \cC\right)_{X \in \cG},
    \end{align}
where $\cP^{\cC} \colon \cD \mapsto \Fun(\cD^{\op}, \cC)$ denotes $\cC$-valued presheaves.
\begin{lemma}
    Let $f\colon X\to Y$ be a morphism in $\cG$. Recall that the left adjoint $\Lan_{(f^*)^{\op}}$ of  
    \[
        \cP^{\cC}(\cG^{\ad}_{/X})\longrightarrow \cP^{\cC}(\cG^{\ad}_{/Y}), \qquad \Phi \mapsto \left((\cG^{\ad}_{/Y})^{\op}\xrightarrow{(f^*)^{\op}}(\cG^{\ad}_{/X})^{\op} \xrightarrow{\Phi} \cC\right) 
    \]
    inherits an oplax symmetric monoidal structure. Then, $\Lan_{(f^*)^{\op}}$ is strong symmetric monoidal.
\end{lemma}
\begin{proof}
     Note that $(f^*)^{\op} \colon (\cG^{\ad}_{/Y})^{\op}\to (\cG^{\ad}_{/X})^{\op} $ is symmetric monoidal for the coCartesian symmetric monoidal structures. By \cite[Proposition 2.21]{Aoki} the pointwise symmetric monoidal structure on  $\cP^{\cC}(\cG^{\ad}_{/X})$ agrees with Day convolution with the cartesian monoidal structure on $\cG^{\ad}_{/X}$. Viewed via Day convolution, the oplax symmetric monoidal structure on 
     \[
       \Lan_{(f^*)^{\op}}\colon \cP^{\cC}(\cG^{\ad}_{/Y})\longrightarrow  \cP^{\cC}(\cG^{\ad}_{/X}) 
     \]
     is symmetric monoidal by \cite[Corollary~3.8]{Nikolaus2016} and \cite[Proposition~6.18 and its proof]{Ching2021}.
\end{proof}
Therefore, we can apply the mate equivalence (\Cref{prop:oplax-vs-lax-lim-for-adjoints-diagram}) to our functor $\cP^{\cC}(\cG^{\ad}_{/-})\colon \cG\to \tCat$  to obtain a diagram
\begin{equation}
    \cG^{\op}\longrightarrow \Catotimes, \qquad X \mapsto \cP^{\cC}(\cG^{\ad}_{/X}) \label{eq:main-functor-presheaves}
\end{equation}
with transition maps given by componentwise left adjoints, i.e.~$\Lan_{(f^*)^{\op}}$, and an equivalence 
\[
    \oplaxlim_{X\in \cG} \cP^{\cC}(\cG^{\ad}_{/X}) \isoo \laxlim_{X\in \cG^{\op}}  \cP^{\cC}(\cG^{\ad}_{/X})
\]
of symmetric monoidal $\infty$-categories. Putting this together with \cref{eq:presheaves-on-arrows-oplax} we arrive at:

\begin{lemma}
    There is a symmetric monoidal equivalence 
    \begin{align}
        \cP^{\cC}(\Ar^{\ad}(\cG)) \isoo \laxlim_{X\in \cG^{\op}} \cP^{\cC}(\cG^{\ad}_{/X}), \quad \Phi \mapsto \left((\cG^{\ad}_{/X})^{\op}\xrightarrow{\fgt} \Ar^{\ad}(\cG)^{\op} \xrightarrow{\Phi} \cC\right)_{X \in \cG^{\op}}.\label{eq:description-lax-limit-presheaves}
    \end{align}
\end{lemma}
\subsection*{The Partially Lax Subcategory}    
By \cref{remark: symmetric monoidal lax limits}, the partially lax limit of the functor $\cP^{\cC}(\cG^{\ad}_{/-})$ in \Cref{eq:main-functor-presheaves} is the symmetric monoidal full subcategory of $\laxlim_{X\in \cG^{\op}} \cP^{\cC}(\cG^{\ad}_{/X})$ consisting of those objects of the lax limit such that the structure maps at admissible morphisms are equivalences.

\begin{observation}\label{obv:left-Kan-along-admissible} Since admissible morphisms are closed under composition, for any admissible morphism $f\colon X\to Y$ the functor 
\[
    f^*\colon \cG^{\ad}_{/Y} \longrightarrow \cG^{\ad}_{/X}
\]
is right adjoint to $f\circ (-)\colon \cG^{\ad}_{/X}\to \cG^{\ad}_{/Y} $. Consequently, $\Lan_{(f^*)^{\op}}$ identifies with
\[
    \cP^{\cC} \left(\cG^{\ad}_{/Y} \right)\longrightarrow \cP^{\cC} \left(\cG^{\ad}_{/X} \right), \qquad \Phi \mapsto \left((\cG^{\ad}_{/X})^{\op}\xrightarrow{f\circ (-)} (\cG^{\ad}_{/Y})^{\op} \xrightarrow{\Phi} \cC\right).
\]
\end{observation}
~\\
We apply the observation to describe the full subcategory defining the partially lax limit: Under the equivalence \eqref{eq:description-lax-limit-presheaves}, the partially lax limit is equivalent to the symmetric monoidal full subcategory of $\cP^{\cC}(\Ar^{\ad}(\cG))$ spanned by those $\Phi \in \cP^{\cC}(\Ar^{\ad}(\cG))$ such that for every admissible $a \colon A \to Y$ and admissible $f \colon Y \to X$ the natural map
    \[  \Phi(f \circ a) \longrightarrow \Phi(a) \]
    is an equivalence, where we use \cref{obv:left-Kan-along-admissible} to identify the source with $\Phi(f \circ a)$. This map is $\Phi$ applied to the morphism $a \to (f \circ a)$ in $\Ar^{\ad}(\cG)$ given by 
    \begin{center}
        \begin{tikzcd}
            A \arrow[r, equal] \arrow[d, "a", swap] & A \arrow[d, "a"]
            \\ Y \arrow[dr, "f", swap, bend right] & Y \arrow[d, "f"]
            \\ & X
        \end{tikzcd}
    \end{center}
    with the constant homotopy filling the diagram. If $\Phi$ satisfies this condition, we say that $\Phi$ is \tdef{partially lax}.
    
    \medskip \noindent Precomposition by the source functor $s\colon \Ar^{\ad}(\cG)\to \cG$ induces a symmetric monoidal functor $$(s^{\op})^*\colon \cP^{\cC}(\cG)\longrightarrow \cP^{\cC}(\Ar^{\ad}(\cG))$$ for the pointwise symmetric monoidal structures. This allows us to state a precise version of Theorem~B.
    \begin{theorem}\label{thm:presheaves-as-lax-lim}
    Let $\cG$ be a small $\infty$-category and $\Ar^{\ad}(\cG)\subseteq \Ar(\cG)$ an admissible class of morphisms. Let $\cC$ be a cocompletely symmetric monoidal $\infty$-category. The symmetric monoidal composite functor 
    \[  
        \cP^{\cC}(\cG)\xrightarrow{(s^{\op})^*} \cP^{\cC}(\Ar^{\ad}(\cG)) \xrightarrow[\eqref{eq:description-lax-limit-presheaves}]{\simeq}  \laxlim_{X\in \cG^{\op}} \cP^{\cC}(\cG^{\ad}_{/X}), \quad \Phi\mapsto\left((\cG^{\ad}_{/X})^{\op}\xrightarrow{\fgt} \cG^{\op} \xrightarrow{\Phi} \cC\right)_{X \in \cG^{\op}}
    \]
    is fully faithful and has essential image the symmetric monoidal full subcategory $\parlaxlim_{X\in \cG^{\op}} \cP^{\cC}(\cG^{\ad}_{/X})$. 
\end{theorem}
\begin{proof}
Consider the functors
    \begin{center}
        \begin{tikzcd}
            \cG \arrow[r, "\const"] \arrow[rr, bend right, equal] & \Ar^{\ad}(\cG) \arrow[r, "s"] & \cG
        \end{tikzcd}
    \end{center}
    with $\const \dashv s$ and adjunction counit evaluated at an admissible morphism $a:A\to B$ depicted by the morphism $\id_{A}\to a$ in the arrow category:
    \[\begin{tikzcd}[cramped]
    	A & A \\
    	A & B
    	\arrow["{\id_{A}}", from=1-1, to=1-2]
    	\arrow[from=1-1, to=2-1]
    	\arrow["a", from=1-2, to=2-2]
    	\arrow["a", from=2-1, to=2-2, swap]
    \end{tikzcd}\]

    \noindent It induces functors
    \begin{center}
        \begin{tikzcd}
            \cP^{\cC}(\cG) \arrow[rr, bend right, swap, equal] \arrow[r, "s^*"] & \cP^{\cC}(\Ar^{\ad}(\cG)) \arrow[r, "\const^*"] & \cP^{\cC}(\cG)
        \end{tikzcd}
    \end{center}
    so that in particular $s^*$ is fully faithful. 
    It suffices to show that the essential image of $s^*$ consists precisely of the partially lax presheaves.
    
    \medskip \noindent First, let $F \in \cP^{\cC}(\cG)$. Then, $s^*F \in \cP^{\cC}(\Ar^{\ad}(\cG))$ lies in the partially lax limit. Indeed, let $f \colon B \to C$ be an admissible morphism. We need to check that the map $F(A) = s^*F(f \circ a) \to s^*F(a) = F(A)$ is an equivalence, but it is the identity. Conversely, we need to check that any partially lax $\Phi \in \cP^{\cC}(\Ar^{\ad}(\cG))$ is in the image of $s^*$. Indeed, the counit of $\const \dashv s$ induces a natural transformation $\varepsilon^* \colon \Phi \Rightarrow s^* \const^* \Phi$ which pointwise evaluates to $\Phi(a) = \Phi(a \circ \id_{s(a)}) \to \Phi(\id_{s(a)})$. This is an equivalence because $\Phi$ is partially lax. So $\Phi \simeq s^* \const^* \Phi$.
\end{proof}

\section{Sheaf Categories as Partially Lax Limits} \label{section: sheafifying}
We will now pass from presheaves to sheaves, and prove \cref{mainthm: A}. Fix a geometric site $(\cG,\cG^{\ad}, \tau)$.
\begin{lemma}\label{lem:characterization-of-sheaves}
    Let $\cC$ be a complete $\infty$-category. 
    A presheaf $\Phi \colon \cG^{\op}\to \cC$ is a sheaf if and only if for all $X\in \cG$ the composite functor
    \begin{center}
        \begin{tikzcd}
            \left(\cG_{/X}^{\ad} \right)^{\op} \arrow[r, hookrightarrow] & \left(\cG_{/X} \right)^{\op} \arrow[r] & \cG^{\op} \arrow[r, "\Phi"] & \cC
        \end{tikzcd}
    \end{center}
    is a sheaf.
\end{lemma}   
\begin{proof}
    In the case of $\cC=\Spc$ this is \cite[Prop. 20.6.3.1]{SAG}.\footnote{To translate his result observe that a presheaf is a sheaf if it satisfies descent for every cover of every $X\in \cG$ and this can be checked in the slice category over $X$.} Lurie's proof goes through for a general target category $\cC$ as it only uses cofinality arguments. 
\end{proof}

\noindent We can sheafify the pointwise monoidal structures on presheaf categories to monoidal structures on sheaf categories. With this structure we now prove our  \cref{mainthm: A}.
\begin{theorem}\label{thm:main-theorem}  Let $\cC$ be a presentably symmetric monoidal $\infty$-category. There is a functor
    \begin{align*}
        \Sh (\cG^{\ad}_{/-};\cC)\colon\cG^{\op} \longrightarrow \Catl^{\otimes}, \qquad X \mapsto \Sh(\cG^{\ad}_{/X};\cC) 
    \end{align*}
    to the $\infty$-category of large symmetric monoidal $\infty$-categories and a natural transformation 
    \begin{align*}
        L\colon \cP^{\cC}(\cG^{\ad}_{/-}) \Longrightarrow \Sh(\cG^{\ad}_{/-};\cC)\qquad \text{of functors } \cG^{\op}\to \Catl^{\otimes}
    \end{align*}
    whose components are $\tau$-sheafification. Moreover, the induced symmetric monoidal functor
    \begin{align*}
        \parlaxlim L\colon \cP^{\cC}(\cG) \longrightarrow \parlaxlim_{X\in \cG^{\op}} \Sh \left(\cG^{\ad}_{/X};\cC \right)
    \end{align*}
    admits a fully faithful right adjoint with essential image the $\tau$-sheaves.
\end{theorem}

\begin{proof}
    We closely follow the argument of \cite[Lemma 4.12]{LNP} in this proof. Let us first ignore the symmetric monoidal structures. Consider the coCartesian unstraightening $\Un^{\co} \cP^{\cC}(\cG^{\ad}_{/-})\to \cG^{\op}$ of the functor $\cP^{\cC}(\cG_{/-}^{\ad}) \colon \cG^{\op} \to \Cat_{\infty}, \ X \mapsto \cP^{\cC}(\cG^{\ad}_{/X})$ that we are taking a partially lax limit over in \Cref{thm:presheaves-as-lax-lim}. Let $\cE\subseteq \Un^{\co} \cP^{\cC}(\cG^{\ad}_{/-})$ denote the full subcategory spanned by $\Sh^{\cC}(\cG^{\ad}_{/X})$ for all $X\in \cG$. 
    
    \medskip \noindent First, the composite $\cE \to \Un^{\co} \cP^{\cC}(\cG^{\ad}_{/-}) \to \cG^{\op}$ is a coCartesian fibration, as shown in the proof of \cite[Lemma 4.12]{LNP}.\footnote{But note that $\cE \to \Un^{\co} \cP^{\cC}(\cG^{\ad}_{/-})$ does not need to preserve all coCartesian edges. More precisely, \cite[Lemma 4.12]{LNP} shows that a coCartesian lift is given by a coCartesian lift in $\Un^{\co} \cP^{\cC}(\cG^{\ad}_{/-})$ followed by the localization map in the fiber.} This uses that for $f \colon X \to Y$ the induced map
    \[ \Lan_{(f^*)^{\op}} \colon \cP^{\cC} \left(\cG_{/Y}^{\ad} \right) \longrightarrow \cP^{\cC} \left(\cG^{\ad}_{/X} \right) \]
    preserves $L$-equivalences, where $L$ denotes sheafification. By an adjunction argument this is equivalent to checking that the associated right adjoint $f^*$ preserves sheaves. This follows because the pullback of a $\tau$-cover is again a $\tau$-cover. The straightening of $\cE \to \cG^{\op}$ is the desired functor $\Sh(\cG^{\ad}_{/-};\cC)$. 
    
    \medskip \noindent Now, the map $\cE \to \Un^{\co} \cP^{\cC}(\cG^{\ad}_{/-})$ fiberwise has a left adjoint. We wish to show that it assembles into a relative left adjoint $\widetilde{L} \colon \Un^{\co} \cP^{\cC}(\cG^{\ad}_{/-}) \to \cE$. By \cite[Proposition D.7]{bachmannHoyois}  it suffices to check for $f \colon X \to Y$ the induced map $\Lan_{(f^*)^{\op}} \colon \cP^{\cC} \left(\cG_{/Y}^{\ad} \right) \longrightarrow \cP^{\cC} \left(\cG^{\ad}_{/X} \right)$
    preserves $L$-equivalences, which we just checked. Moreover, $\widetilde{L}$ is automatically a map of coCartesian fibrations, see e.g.~the proof of \cite[Proposition 2.2.5(2)]{hilman2024parametrisedpresentability}. By straightening $\widetilde{L}$ we have thus built the natural transformation $L$ claimed in the theorem.
    
    \medskip \noindent Furthermore, the inclusion $\cE \hookrightarrow \Un^{\co} \cP^{\cC}(\cG^{\ad}_{/-})$ preserves coCartesian lifts of admissibles. For this, we need to show that for every admissible morphism $f \colon X \to Y$ the functor $\Lan_{(f^*)^{\op}}$ preserves $\tau$-sheaves. By \Cref{obv:left-Kan-along-admissible} this amounts to showing that if $\Phi\colon (\cG_{/Y}^{\ad})^{\op}\to \cC $ is a $\tau$-sheaf, then the composite 
    \begin{center}
        \begin{tikzcd}
            \left(\cG^{\ad}_{/X} \right)^{\op} \arrow[r, "f \circ -"] & \left( \cG^{\ad}_{/Y} \right)^{\op} \arrow[r, "\Phi"] & \cC
        \end{tikzcd}
    \end{center}
    is a $\tau$-sheaf as well. This holds as $\tau$-covers in $\cG_{/Z}^{\ad}$ are detected in $\cG$ for all $Z\in \cG$. 
    
    \medskip \noindent By \cref{remark: berman theorem} the partially lax limits of $\Sh(\cG^{\ad}_{/-};\cC)$ and $\cP^{\cC}(\cG^{\ad}_{/-})$ are computed by the $\infty$-categories of those sections of the respective cocartesian fibrations that send admissible morphisms to coCartesian morphisms and $\parlaxlim L$ is given by postcomposition with the coCartesian edges preserving functor $\tilde{L}$. Therefore, postcomposition with the inclusion $\cE\into  \Un^{\co} \cP(\cG^{\ad}_{/-})$ gives a fully faithful right adjoint to $\parlaxlim L$. 
    
    \medskip \noindent So the essential image of the right adjoint to $\parlaxlim L$ consists of those presheaves $\Phi\colon \cG^{\op}\to \cC$ satisfying for all $X\in \cG$ that the composite functor
    \begin{center}
        \begin{tikzcd}
            \left(\cG_{/X}^{\ad} \right)^{\op} \arrow[r, hookrightarrow] & \left(\cG_{/X} \right)^{\op} \arrow[r] & \cG^{\op} \arrow[r, "\Phi"] & \cC
        \end{tikzcd}
    \end{center}
    is a sheaf. \Cref{lem:characterization-of-sheaves} implies that the essential image of the right adjoint consists precisely of the $\tau$-sheaves.
    
    \medskip \noindent Finally, we apply \cite[Corollary 4.14]{LNP} to $F\coloneqq \cP^{\cC}(\cG^{\ad}_{/-})$ and $\left\{GX\coloneqq \Sh^{\cC}(\cG^{\ad}_{/X}) \right\}_{X\in \cG}$ to lift $L$ to a symmetric monoidal transformation.
\end{proof}

\begin{remark}\label{thm:main-theorem-hyper}
    The analogous statement of \Cref{thm:main-theorem} holds for sheaves replaced by hypersheaves. Indeed, the proof of \cref{thm:main-theorem} goes through in this setting.
\end{remark}

\begin{remark}\label{rem:modules}
By \cite[Theorem ~5.10]{LNP}, we can also pass to modules over any $R\in \CAlg(\Sh(\cG;\cC))$ in \Cref{thm:main-theorem} to obtain a symmetric monoidal equivalence 
\begin{align*}
    \Mod_R \left(\Sh(\cG;\cC) \right)\isoo \parlaxlim_{X \in \cG^{\op}} \Mod_{R_X}\left(\Sh(\cG^{\ad}_{/X};\cC) \right),
\end{align*}
where $R_X$ denotes the restriction of $R$ to $\Sh(\cG^{\ad}_{/X};\cC)$.
\end{remark}

\begin{corollary} \label{corollary: R of Sh}
     Let $(\cG, \cG^{\ad}, \tau)$ be a geometric site with terminal object $* \in \cG$ and $\mathcal{R} \colon \mathbb{C} \to \cD$ be a lax-invariant functor, where $\mathbb{C} \in \{\CAlg(\Pr^L), \CAlg(\Pr^L_{\mathrm{st}}) \}$ and $\cD$ is any $\infty$-category. Let $\cC \in \mathbb{C}$. Then, the restriction functor induces an equivalence
     \[ \cR(\Sh(\cG;\cC)) \xrightarrow{\ \simeq \ } \cR \left( \Sh(\cG^{\ad}_{/*};\cC) \right). \]
\end{corollary}

\begin{proof} 
     Apply \cref{mainthm: A}. The limit is then computed on the initial object $* \in \cG^{\op}$.
\end{proof}

\noindent We now apply \cref{mainthm: A} to our condensed example (\cref{example: condensed geometric site}).

\begin{corollary} \label{corollary: Cond as lax limit}
     Let $\cC$ be a presentably symmetric monoidal $\infty$-category. Then, there is an equivalence 
     \[ \Cond(\cC) \simeq \parlaxlim_{X \in \ExtrDisc^{\op}} \Sh(X;\cC) \]
     of symmetric monoidal $\infty$-categories, where the open embeddings are marked. In particular, there is an equivalence $\cR(\Cond(\cC)) \simeq \cR(\cC)$ for lax-invariant functors $\cR$.
\end{corollary}

\begin{proof}
     A geometric site structure on $\ExtrDisc$ was exhibited in \cite[Theorem A]{rasekhzhu2026fracturedstructurescondensedmathematics} and moreover the associated slice categories were computed in \cite[Theorem B]{rasekhzhu2026fracturedstructurescondensedmathematics}. Combined with \cref{mainthm: A} we arrive at the desired. The last part follows from \cref{corollary: R of Sh}.
\end{proof}

\noindent We obtain similar descriptions in the setting of \cref{example: alggeo geometric site} and \cref{example: geometry geometric site}.

\section{Gros étale exodromy} \label{section: gros exodromy}
We recall that the \tdef{condensed Galois $\infty$-category} of $X$ from \cite{barwick2020exodromy, haine2025condensedhomotopytypescheme, mair2026galoiscategoriescondensedcontractible, debruyn2026condensedproofproetaleetale} is
\[ \mdef{\Gal(X)} \colon \ExtrDisc^{\op} \longrightarrow \Cat_{\infty}, \ S \mapsto \Fun^*_{\mathrm{loc} \, \mathrm{coh}}(\Sh(X_{\et}), \Sh(S)), \]
which for $S \in \ExtrDisc^{\op}$ consists of locally coherent geometric morphisms $\Sh(S) \to \Sh(X_{\et})$, see \cite[2.1.6]{debruyn2026condensedproofproetaleetale}. It is functorial in $X$ because maps of schemes induce locally coherent morphisms between their associated étale topoi. This object plays a central role in Barwick--Glasman--Haine's exodromy philosophy \cite{barwick2020exodromy}. Before stating it, we need to recall some notation about condensed $\infty$-categories.
\begin{enumerate}
    \item For two functors $\cC, \cD \colon \ExtrDisc^{\op} \to \Cat_{\infty}$ we denote by
\[ \mdef{\Fun^{\cts}(\cC, \cD)} \coloneqq \Nat(\cC,\cD) \, \simeq \, \lim_{(S \to S') \in \Tw(\ExtrDisc^{\op})^{\op}} \Fun(\cC(S), \cD(S')) \]
the $\infty$-category of \tdef{continuous functors} from $\cC$ to $\cD$, see \cite[Definition 3.1.1]{debruyn2026condensedproofproetaleetale}. 
    \item Let $\cC$ be a compactly assembled presentable $\infty$-category. The condensed $\infty$-category $\mdef{\ul{\cC}}$ is defined by
    \[ \ul{\cC} \colon \ExtrDisc^{\op} \longrightarrow \Cat_{\infty}, \ S \mapsto \Sh(S;\cC), \]
    see \cite[Definition 3.2.4]{debruyn2026condensedproofproetaleetale} for more details.
\end{enumerate}
Work of van Dobben de Bruyn provides a (petit version) of étale exodromy:

\begin{fact}[{\cite[Theorem 2]{debruyn2026condensedproofproetaleetale}}] \label{fact: remy}
    Let $X$ be a scheme, and let $\cC$ be a compactly assembled presentable $\infty$-category. Then, the pro-étale exodromy correspondence restricts to an étale exodromy correspondence
    \[ \mathrm{Ex}^{\et} \colon \Sh^{\post}(X_{\et}; \cC) \xrightarrow{\ \simeq \ } \Fun^{\cts}(\Gal(X), \ul{\cC}). \]
    Here, the pro-étale exodromy map is constructed in the proof of \cite[Theorem 1]{debruyn2026condensedproofproetaleetale}.
\end{fact}

\begin{remark}
    One can check by hand that $\mathrm{Ex}^{\et}$ is natural in $X$, by first checking a $1$-categorical naturality and then applying functorial $\infty$-categorical machinery.
\end{remark}

\noindent We will now use \cref{mainthm: A} to deduce a gros version of this result. Throughout, let us denote by $\Sch$ the (1-)category of $\kappa$-small qcqs schemes.

\begin{definition} \label{def: gros Galois category}
    Let the \tdef{gros Galois $\infty$-category} be 
    \[ \mdef{\GAL} \coloneqq \laxcolimdag_{X \in \Sch} \Gal(X) \in \Fun \left(\ExtrDisc^{\op}, \Cat_{\infty} \right), \]
    where we mark the étale morphisms, and the partially lax colimit is taken pointwise.
\end{definition}

\begin{remark}
    For our purposes, it is not needed whether $\GAL$ is a condensed $\infty$-category, i.e.~a sheaf, and we do not know whether this is true.
\end{remark}

\begin{theorem} \label{theorem: gros etale exodromy}
    Let $\cC$ be a compactly assembled presentable $\infty$-category. Then, there is a preferred étale exodromy equivalence
    \[ \Sh^{\post}(\Sch_{\et}; \cC) \xrightarrow{\ \simeq \ } \Fun^{\cts}(\GAL, \ul{\cC}) \]
    of $\infty$-categories. 
\end{theorem}

\begin{proof}
     It's enough to prove that the $n$-truncated version
     \[ \Sh(\Sch_{\et};\An_{\leq n}) \longrightarrow \Fun^{\cts}(\GAL, \ul{\An}_{\leq n}) \]
     is an equivalence. As in \cite[Theorem 3.2.6]{debruyn2026condensedproofproetaleetale} we can pass to the limit and tensor by $\cC$ to finish. 
     
     \medskip \noindent Applying $\parlaxlim$ to \cref{fact: remy} combined with our \cref{mainthm: A} we deduce
     \[ \Sh(\Sch_{\et};\An_{\leq n}) \xrightarrow{\ \simeq \ } \parlaxlim_{X \in \Sch^{\op}} \Sh(X_{\et}; \An_{\leq n}) \xrightarrow{\ \simeq \ } \parlaxlim_{X \in \Sch^{\op}} \Fun^{\cts}(\Gal(X), \ul{\An}_{\leq n}), \]
     so we are left to identify the right side. We do this here:
     \begin{align*}
         \parlaxlim_{X \in \Sch^{\op}} \Fun^{\cts}(\Gal(X), \ul{\An}_{\leq n}) &\simeq \lim_{(X \to X') \in \Tw(\Sch^{\op})^{\op}} \Fun \left(|\Sch_{/X}^{\op \dagger}|, \Fun^{\cts}(\Gal(X')(-), \ul{\An}_{\leq n}) \right)
         \\ &\simeq \lim_{X \to X'} \lim_{(S \to S') \in \Tw(\ExtrDisc^{\op})^{\op}} \Fun \left(|\Sch^{\op \dagger}_{/X}|, \Fun(\Gal(X')(S), \ul{\An}_{\leq n}(S')) \right)
         \\ &\simeq \lim_{X \to X'} \lim_{S \to S'} \Fun \left(\Gal(X')(S) \times |\Sch^{\op \dagger}_{/X}|, \ul{\An}_{\leq n}(S') \right)
         \\ &\simeq \lim_{S \to S'} \lim_{X \to X'} \Fun \left(\Gal(X')(S) \times |\Sch^{\op \dagger}_{/X}|, \ul{\An}_{\leq n}(S') \right)
         \\ &\simeq \lim_{S \to S'} \Fun \left( \colim_{(X \to X') \in \Tw(\Sch^{\op})} \Gal(X')(S) \times |\Sch^{\op \dagger}_{/X}|, \ul{\An}_{\leq n}(S') \right)
         \\ &\simeq \Fun^{\cts} \left( \laxcolimdag_X \Gal(X), \ul{\An}_{\leq n} \right),
     \end{align*}
     as desired.
\end{proof}

\begin{remark} \label{remark: exodromy base scheme}
    The same proof goes through for $\Sch_T$ with $\kappa$-small schemes qcqs over an arbitrary base scheme $T$ instead of $\Sch$. So let $\mdef{\GAL_T} \coloneqq \laxcolimdag_{X \in \Sch_T} \Gal(X)$, then there is an equivalence
    \[ \Sh^{\post}(\Sch_{T,\et};\cC) \simeq \Fun^{\cts}(\GAL_T, \ul{\cC}) \]
    of $\infty$-categories.
\end{remark}

\bibliographystyle{alpha}
\bibliography{main.bib}
\end{document}